\documentclass[12pt]{amsart}
\usepackage{amsmath}
\usepackage{amssymb}
\usepackage{epsfig}
\usepackage{graphicx}
\usepackage{xcolor}
\numberwithin{equation}{section}

\usepackage{tikz}
\newcommand*\circled[1]{\tikz[baseline=(char.base)]{
            \node[shape=circle,draw,inner sep=2pt] (char) {#1};}}

\usepackage[all,2cell]{xy}  
\usepackage{float}

\calclayout
\allowdisplaybreaks[3]

\DeclareFontFamily{U}{mathb}{\hyphenchar\font45}
\DeclareFontShape{U}{mathb}{m}{n}{
      <5> <6> <7> <8> <9> <10> gen * mathb
      <10.95> mathb10 <12> <14.4> <17.28> <20.74> <24.88> mathb12
      }{}
\DeclareSymbolFont{mathb}{U}{mathb}{m}{n}
\DeclareMathSymbol{\righttoleftarrow}{3}{mathb}{"FD}

\newcommand{\eqto}{\stackrel{\lower1.5pt\hbox{$\scriptstyle\sim\,$}}\to}
\newcommand{\eqdashto}{\stackrel{\lower1.5pt\hbox{$\scriptstyle\sim\,$}}\dashrightarrow}
\newcommand{\actsfromleft}{\mathrel{\reflectbox{$\righttoleftarrow$}}}
\newcommand{\actsfromright}{\righttoleftarrow}

\usepackage[all,2cell]{xy}  

\newcommand{\xycenter}[1]{
	\begin{center}
	\mbox{\xymatrix{#1}}
	\end{center}
	}

\theoremstyle{plain}
\newtheorem{prop}{Proposition}

\newtheorem{theo}[prop]{Theorem}

\newtheorem*{theo*}{Theorem}

\newtheorem{lemm}[prop]{Lemma}

\theoremstyle{definition}
\newtheorem{defi}[prop]{Definition}

\newtheorem{rema}[prop]{Remark}

\newtheorem{exam}[prop]{Example}

\makeatother
\makeatletter

\newcommand{\bA}{\mathbb A}
\newcommand{\bC}{\mathbb C}
\newcommand{\bF}{\mathbb F}

\newcommand{\bP}{\mathbb P}
\newcommand{\bQ}{\mathbb Q}

\newcommand{\bZ}{\mathbb Z}

\newcommand{\cE}{\mathcal E}

\newcommand{\cK}{\mathcal K}

\newcommand{\cO}{\mathcal O}

\newcommand{\cU}{\mathcal U}

\newcommand{\rD}{\mathsf D}

\newcommand{\rH}{\mathrm H}
\newcommand{\rK}{\mathrm K}

\newcommand{\fA}{\mathfrak A}
\newcommand{\fD}{\mathfrak D}
\newcommand{\fS}{\mathfrak S}

\newcommand{\PGL}{\operatorname{PGL}}

\newcommand{\Burn}{\operatorname{Burn}}

\newcommand{\Gr}{\operatorname{Gr}}
\newcommand{\Pic}{\operatorname{Pic}}

\newcommand{\Aut}{\operatorname{Aut}}

\newcommand{\Lin}{\mathfrak{L}\mathrm{in}}
\newcommand{\PLin}{\mathfrak{P}\mathfrak{L}\mathrm{in}}
\newcommand{\SLin}{\mathfrak{S}\mathfrak{L}\mathrm{in}}

\newcommand{\lra}{\longrightarrow}

\author[B\"ohning]{Christian B\"ohning}
\address{Mathematics Institute\\
Zeeman Building\\
University of Warwick\\
Coventry CV4 7AL\\
UK}
\email{C.Boehning@warwick.ac.uk}

\author[von Bothmer]{Hans-Christian Graf von Bothmer}
\address{Fachbereich Mathematik\\
Bundesstraße 55\\
20146 Hamburg\\
Germany}
\email{hans.christian.v.bothmer@math.uni-hamburg.de}

\author[Tschinkel]{Yuri Tschinkel}
\address{Courant Institute\\
                New York University \\
                New York, NY 10012 \\
                USA }
\email{tschinkel@cims.nyu.edu}

\address{Simons Foundation\\
160 Fifth Avenue\\
New York, NY 10010\\
USA}

\title[Equivariant birational types]{Equivariant birational types and derived categories}

\begin{document}
\date{\today}

\begin{abstract}
We investigate equivariant birational geometry of rational surfaces and threefolds 
from the perspective of derived categories.  
\end{abstract}

\maketitle

\section{Introduction}
\label{sect:intro}

Let $X$ be a smooth projective variety over an algebraically closed field $k$, of characteristic zero. Assume that $X$ is equipped with a regular, generically free, action of a finite group $G$. 
A major topic in birational geometry is to understand 
equivariant birational types, e.g., to decide 
whether or not $X$ is 
\begin{itemize}
\item 
{\em (projectively) linearizable}, i.e., 
equivariantly birational to projective space, with a (projectively) linear action of $G$, or
\item {\em stably (projectively) linearizable}, i.e., (projectively) linearizable
after taking a product with $\bP^m$, for some $m$, with trivial action on the
second factor.
\end{itemize}
One of the motivations is the analogy of this theory with birational geometry over nonclosed ground fields and, in particular, with the central problem of (stable) rationality over such fields, where the role of $G$ is taken by the absolute Galois group of the ground field, acting on geometric objects.

Various tools have been developed to distinguish equivariant birational types, e.g., cohomology, derived categories, and more recently, equivariant Burnside groups (see \cite{HT-stable}). 
In this note, we investigate the interactions between different perspectives on the (stable) linearizability problem. We focus on low-dimensional examples, in particular, Del Pezzo surfaces, rational Fano threefolds and fourfolds. We explore the compatibility of group actions with standard (stable) rationality constructions and conjectures, and produce new examples of stably linearizable but nonlinearizable actions.

In detail, in Section~\ref{sect:equi}, we discuss basic notions of equivariant birational geometry, classical invariants of $G$-actions on varieties, as well as the recently developed Burnside formalism \cite{BnG}.  We present applications of a multilinear algebra construction, Proposition~\ref{prop:stable-rat}, to exhibit new examples of nonlinearizable but stably linearizable actions, e.g., we show in Example~\ref{exam:a5} that, 
for $G=\fA_5$, the $G$-birationally rigid, and thus not linearizable, quintic Del Pezzo threefold is stably linearizable.  

In Section~\ref{sect:derived}, we 
study exceptional sequences in derived categories, in presence of $G$-actions, and their connections with classical invariants. In Section~\ref{sect:dp}, 
we prove

\begin{theo*}
    A smooth projective rational $G$-surface that is linearizable has a full $G$-equivariant exceptional sequence. 
\end{theo*}

The proof relies on the classification of finite subgroups in the Cremona group of \cite{DI}, and subsequent developments in equivariant geometry of rational surfaces. Over nonclosed fields the situation was investigated in \cite{auel-bernardara} and \cite{BD-cat}, in particular, we view this theorem as an analog of \cite[Corollary 1]{auel-bernardara}. However, we also give an example, in the equivariant context, where the analog of \cite[Theorem 1]{auel-bernardara} fails. 

In Section~\ref{sect:three}, we turn to Fano threefolds. 
For quintic Del Pezzo threefolds, with give examples of nonlinearizable actions of finite groups $G$ with derived categories admitting full exceptional sequences of $G$-linearized objects. 
This disproves the equivariant analog of the well-known conjecture that a smooth projective variety with a full exceptional sequence over the ground field should be rational. The corresponding $G$-actions are stably linearizable.

For Fano threefolds of genus 7, we show that there are nonlinearizable actions in presence of $G$-invariant semiorthogonal decompositions, with pieces equivalent, as $G$-categories, to derived categories of $G$-varieties of codimension $\ge 2$.

\

\noindent
{\bf Acknowledgments:} 
For the purpose of open access, the first author has applied a Creative Commons Attribution (CC-BY) licence to any Author Accepted Manuscript version arising from this submission. 
The third author was partially supported by NSF grant 2301983.

\section{Equivariant geometry}
\label{sect:equi}

\subsection*{Terminology}
Throughout, $G$ is a finite group. We consider generically free regular actions of $G$ on irreducible algebraic varieties over $k$, an algebraically closed field of characteristic zero, and refer to such varieties as $G$-varieties. We write
$$
X\sim_G Y
$$
if the $G$-varieties $X,Y$ are equivariantly birational, and 
introduce subcategories of the category of $G$-varieties:
\begin{itemize}
    \item $\Lin$ - $G$-linearizable, 
    \item $\PLin$ - projectively $G$-linearizable, 
    \item $\SLin$ - stably $G$-linearizable. 
    \end{itemize}

A basic result in $G$-birational geometry is equivariant resolution of singularities and weak factorization: $G$-birational varieties
are related via blowups and blowdowns, with centers in smooth $G$-stable subvarieties. Moreover, after a sequence of such blowups, one can reach a {\em standard model} $\tilde{X}\to X$ such that on $\tilde{X}$ all stabilizers are abelian, and $G$-orbits of divisors with nontrivial stabilizers are smooth, i.e., for every such $D$ and $g\in G$, the intersection $(D\cdot g)\cap D$ is either all of $D$ or empty.

\subsection*{Classical invariants}
The $G$-action on $X$ induces actions on cohomology groups, and in particular on the Picard group $\Pic(X)$. 
There is an exact sequence, see, e.g., \cite[\S 3]{KT-dp}
\[
\!\!\!\!\xymatrix{
& \mathrm{Pic}(X, G) \ar[r] & \mathrm{Pic}(X)^G \ar[r]^{\!\!\delta_2} & \rH^2 (G, k^{\times} ) \\
\ar[r] & \mathrm{Br} ([X/G]) \ar[r] & \rH^1 (G, \mathrm{Pic}(X)) \ar[r]^{\,\,\delta_3} & \rH^3 (G, k^{\times} ),
}
\]
where: 
\begin{itemize}
\item $\mathrm{Pic}(X, G)$ is the group of isomorphism classes of $G$-linearized line bundles, and $\mathrm{Pic}(X)^G$ the group of $G$-invariant line bundles
\item $[X/G]$ is the quotient stack, $\mathrm{Br} ([X/G])$ its Brauer group, and
\end{itemize}
Both $\delta_2$ and $\delta_3$ are zero when there are $G$-fixed points; these give sections of the map $[X/G]\to BG$.  
Other frequently  studied obstructions to (stable) linearizabilty are:
\begin{itemize}
    \item $\mathrm{Am}(X,H)$, the Amitsur invariant, i.e., the image of 
    $$
    \delta_2:  \Pic(X)^H\to \rH^2(H,k^\times), \quad H\subseteq G,
    $$
    \item {\bf (H1)}: $\rH^1(H,\Pic(X)) =\rH^1(H,\Pic(X)^\vee)=0$, $H\subseteq G$,
    \item {\bf (SP)}: $\Pic(X)$ is a stable $G$-permutation module.
 \end{itemize}
 If $X\in \PLin$ or $\Lin$, then {\bf (H1)} and {\bf (SP)} hold; when $X\in \Lin$ then 
 the Amitsur invariant vanishes.

\subsection*{Burnside formalism}
Let $G$ be a finite group, acting on $X$, a standard model for the action.  
On such a model one computes the class of the $G$-action in the {\em Burnside group}:
\begin{equation}
\label{eqn:class}
[X\actsfromright G] = \sum_{F,H} (H,Y\actsfromleft k(F), \beta) \in \Burn_n(G),  \, n=\dim(X),
\end{equation}
as a sum of {\em symbols}, recording ($G$-orbits of) irreducible subvarieties $F\subset X$ with nontrivial generic stabilizer $H$, together with the induced action of a subgroup $Y\subseteq Z_G(H)/H$ on the function field $k(F)$ and the collection $\beta$ of weights of $H$ in 
the normal bundle to $F$ (all defined up to conjugation in $G$).
In particular, this sum contains the trivial summand
$$
(1, G\actsfromleft k(X), ()),
$$
The symbols are subject to explicit relations so that the class 
\eqref{eqn:class} is 
an equivariant birational invariant (see \cite{BnG}, \cite{HKT-small} for definitions and examples). 
The trivial summand does not participate in relations; we say that the $G$-action on $X$ has trivial Burnside class if
$$
[X\actsfromright G] = (1, G\actsfromleft k(X), ())
$$
in $\Burn_n(G)$. 
{\em Incompressible divisorial symbols} (modulo conjugation relation), in the terminology of \cite[Definition 3.3]{KT-vector}, generate, freely, a direct summand of $\Burn_n(G)$; in many situations, it suffices to compare their contribution to $[X\actsfromright G]$ to 
distinguish $G$-actions up to equivariant birationality, see \cite[Section 3.6]{TYZ}. 

The paper \cite{KT-vector} provides an algorithm for the computation of $[\bP(V)\actsfromright G]$ for linear and projective linear actions of a finite group $G$; this algorithm has been implemented in {\tt Magma}, see \cite{TYZ}. 
While the formalism and the computations can be involved, incompressible divisorial symbols allow to quickly show nonlinearizability of some actions.
Indirectly, they also lead to constraints on possible actions:

\begin{exam}
\label{exam:x22}
Let $X\subset \bP^n$ be a prime (smooth) Fano threefold of index 1, in its anticanonical embedding. 
Let $\sigma\in \PGL_{n+1}$ be an involution preserving a hyperplane. 
Does $\sigma$ preserve $X$? 

If so, we would have $X^\sigma = S$, 
a surface, yielding a symbol
$$
(\langle \sigma\rangle, 1\actsfromleft  k(S), (1))\in \Burn_3(C_2).
$$
Generically, $S$ would be a K3 surface, the symbol incompressible, and thus the action not linearizable. 

On the other hand, consider smooth Fano threefolds $X=X_{22}$ of genus 12. 
We know that $G$-actions on $X$ are linearizable, if there is a (sufficiently general) fixed point; in the arithmetic setup this is discussed in \cite[Theorem 5.17]{Kuz-Pro}. 

This tension can be reconciled, in fact, $X$ cannot carry such involutions. We sketch an argument:
According to Mukai, cf. \cite{Schreyer01}, $X$ can be constructed as follows: start with a $7$-dimensional vector space $V$, a $3$-dimensional vector space $U$, and a linear map $\eta\colon \wedge^2(V) \to U^*$. Dually, this arises from a linear map $\eta^*\colon U \to \wedge^2(V^*)$. Consider the Grassmannian $\mathrm{Gr}(3, V) \subset \bP( \wedge^3(V))$, and for the universal subbundle $\mathcal{U}$ notice that $\rH^0 (\mathrm{Gr}(3, V), \mathcal{U}^*)=\wedge^2 (V^*)$. The zeros of the sections in $U$ on $\mathrm{Gr}(3, V)$ yield $X$. Equivalently, we have a linear map
\[
\wedge^3(V) \to V \otimes U^*,
\]
induced by wedging elements in $\wedge^3(V)$ with elements in $U$, and the kernel $K$ is a $14$-dimensional subspace of $\wedge^3(V)$ such that 
$$
X= \mathrm{Gr}(3, V) \cap \bP(K).
$$
We can view $\eta^*$ as an element in $\wedge^2 (V^*) \otimes U^*$, i.e., a skew-symmetric $7\times 7$-matrix with entries in $U^*$. The $6\times 6$ Pfaffians of this matrix define an Artinian Gorenstein module of codimension $3$ over $k[U^*]$, with dual socle generator a quartic $F$, see \cite[Theorem 2.6 and its proof]{Schreyer01}. Conversely, the datum of this quartic or Artinian Gorenstein module allows to reconstruct the skew-symmetric $7\times 7$-matrix with entries in $U^*$ by considering the middle map in the Buchsbaum-Eisenbud resolution of the module. 

By \cite[Theorem 6.1]{Schreyer01}, the Scorza quartic $S_F$ covariantly associated to $F$ is isomorphic to the Hilbert scheme of lines $F_1 (X)$ in $X$, and by \cite[Appendix, Claim A.1.1., (by Prokhorov)]{DM22}, the automorphisms of $X$ embed injectively into those of $F_1 (X)$. 

Now assume that $G=C_2$ is acting faithfully on $X$, through the $G$-representations $U$ and $V$ and an equivariant map $\eta$, as described above. Then $G$ embeds into the automorphisms of the quartic $C$, and acts faithfully on $U$ with weights $(1,1,0)$ or $(0,0,1)$. Then \cite[Theorem 2.6]{Schreyer01} implies that $V$ can be recovered as the kernel of the linear map 
\[
\mathrm{Sym}^3(U) \to U 
\]
given by contracting with the equation of the $C_2$-invariant quartic $F$. This map is equivariant (possibly after tensoring the target $U$ with a sign). 
After that, we recover $K$ as the kernel of the $G$-equivariant map 
\[
\wedge^3(V) \to V \otimes U^*,
\]
above. Working through all sign combinations, one verifies that   $G$ cannot act on $K$ with all weights but one equal to each other.   
\end{exam}

\begin{exam}
\label{exam:cubic}
Let $X\subset \bP^5$ be a smooth cubic fourfold, with an action of $G=C_m$, with weights $(0,0,0,0,0,b)$; note that only $m=2,3$ are possible, under this assumption. Assume that the divisor $D\subset X$, given by the vanishing of the last coordinate, is smooth. Then $X\notin \PLin$. 

Indeed, the corresponding symbol 
$(C_m,1\actsfromleft k(D),(b))$ is incompressible, since $D$ is not birational to $S\times \bP^1$, for any surface $S$, and does not appear in classes of linear actions, 
(see \cite[Corollary 6.1]{TYZ}). 

For $m=2$, \cite[Theorem 1.2(2)]{Marcubic}
shows that a {\em very general} $X$ carrying such an involution does not have an associated K3 surface and is expected to be nonrational; and, in particular, the action would not be linearizable.  
The same argument applies for $m=3$. 
\end{exam}

\subsection*{Pfaffians and Grassmannians}
\label{sect:grass}

In \cite[Section 7]{BBT23}, we have used a construction from multilinear algebra, the Pfaffian construction, to exhibit nonlinearizable but stably linearizable actions of finite groups on rational varieties, e.g., rational cubic fourfolds. The starting point is a Pfaffian variety 
$$
X:=\mathrm{Pf}(W)\cap \bP(L),
$$
where $V$ is a vector space of dimension $n=2m$ and $L\subset \wedge^2(W)$ a linear subspace of dimension $n$. Then there is a diagram 

\centerline{
\xymatrix{  & \ar[dl]_{p} \bP(\cK_X) \ar[dr]^{q} & \\
X &     & \bP(W^*)
}
}

\

\noindent
where $\cK_X$ is a vector bundle of rank 2 and $q$ is birational. In presence of group actions, choosing a $G$-representation $W$ and a subrepresentation $L$, one obtains, under suitable genericity assumptions, 
an equivariant birationality:
$$
X\times \bP^1\sim_G \bP(W^*), 
$$
with trivial action on the second factor.

\begin{exam}
\label{cubic:new}
Let $G=C_5\rtimes \fD_{15}\rtimes C_3$, {\tt{GapID}}(450,24). 
It acts generically freely on the singular (toric) cubic fourfold $X\subset \bP^5$ with equation
$$
x_1x_3x_5+x_2x_4x_6=0,
$$
a degeneration of the Pfaffian cubic considered in \cite[Example 14]{BBT23}. The $G$-action on $X$ is not linearizable, as $G$ 
does not have faithful representations of dimension $<6$. 
The Pfaffian construction applies: by \cite[Corollary 13]{BBT23}, the $G$-action on $X\times \bP^1$, with trivial action on the second factor, is linearizable.
\end{exam}

Here, we present another such construction, applicable to subvarieties of Grassmannians. 
Linear sections of Grassmannians admit tautological stable rationality constructions, that we now describe: Let $W$ be an $n$-dimensional vector space over $k$ and $\Gr(2,W)$ the Grassmannian of planes in $W$. Let $V\subset \wedge^2(W)$ be a linear subspace of codimension $r$. 
Put 
$$
X:=\Gr(2,W)\cap \bP(V)
$$
and consider the diagram

\centerline{
\xymatrix{  & \ar[dl]_{p} \bP(\cU_X) \ar[dr]^{q} & \\
X &     & \bP(W),
}
}

\

\noindent
where $\cU_X$ is the restriction of the universal vector bundle over $\Gr(2,W)$ to $X$. This yields a stable rationality construction, as both $p$ and $q$ are vector bundles, in the indicated range of dimensions. 

\begin{prop}
\label{prop:stable-rat}
Let $k$ be an algebraically closed field of characteristic zero and $G$ a finite group. Let $W$ be an $n$-dimensional representation of $G$ over $k$ and $V\subset \wedge^2(W)$ a subrepresentation of codimension $r\le n-2$   
such that the $G$-actions on $\bP(W)$ and $\bP(V)$ are generically free. Assume that  
\begin{itemize}
    \item[($\ast$)] $X$ is irreducible of dimension $\dim(\Gr(2,n))-r = 2(n-2)-r$.  
\end{itemize}
Then 
$$
X\times \bP^1 \sim_G \bP(W)\times \bP^{n-2-r},
$$
with trivial actions on the second factors. 
\end{prop}

\begin{proof}
By the No-name Lemma, $X\times \bP^1 \sim_G \bP(\cU_X)$. 
Note that each fiber of $q$ is nonempty: indeed, the fiber over $[w] \in \bP(W)$ is $\bP(w\wedge W)\cap \bP(V)$, which has dimension $\ge n-2-r \ge 0$, the last inequality by the assumption $r\le n-2$. By assumption $(\ast)$, it follows that for generic $w\in W$, one has 
    $$
    \dim(\bP(w\wedge W)\cap \bP(V))= n-2-r.
    $$
Thus $\bP(\cU_X)$, which is irreducible since $X$ is, is generically the projectivization of a $G$-vector bundle over $\bP(W)$ via $q$. Another application of the No-name Lemma yields the result. 
\end{proof}


\begin{rema}
If we drop assumption $(\ast)$, but keep assuming $r\le n-2$, then the construction of Proposition \ref{prop:stable-rat} still yields stable linearizability for the unique component of $X$ such that the restriction of $\bP(\mathcal{U}_X)$ to it dominates $\bP(W)$. But proving nonlinearizability of such a component of $X$ is usually difficult, unless we assume a condition similar to $(\ast)$, a priori. 
\end{rema}

This construction works also over nonclosed fields. However,
there one does not gain new insights:
by \cite[Theorem 2.2.1]{Xu}, if $r\le n-2$ and $X$ is smooth then $X$ is already rational over $k$. The proof uses the same diagram, restricted to a codimension one linear subspace $\Pi$ in $\bP(W)$, exhibiting $X$ as birational to a vector bundle over $\Pi$, thus rational over $k$.  
In presence of group actions, this can fail, e.g., if $W$ does not admit a subrepresentation of codimension one! 
This yields many examples of nonlinearizable but stably linearizable actions.

\begin{exam}
\label{exam:s5}
Let $G=\fS_5$, and $W:=W_5$ its 5-dimensional representation. We have a decomposition
$$
\wedge^2(W)=W_6\oplus W_4,
$$
as representations. 
When $V=W_6$ is the 6-dimensional subrepresentation, 
$$
S:=\Gr(2,W)\cap \bP(V)
$$ 
is the del Pezzo surface of degree 5. 
It is easy to see that the induced $G$-action on $S$ is not linearizable, indeed, $\fS_5$ does not admit a linear action on $\bP^2$. Even the restriction to $\fA_5\subset \fS_5$ is not linearizable,  
see, e.g., \cite[Theorem 6.6.1]{CS}.

Note that the assumptions of 
Proposition~\ref{prop:stable-rat} are {\em not} fulfilled, we have $n=5$ and $r=4$, rather than $r\le 3$. 
Nevertheless, by \cite[Proposition 4.7]{pro}, $S\times \bP^1$ is $\fS_5$-equivariantly birational to the Segre cubic threefold, with the action of the {\em nonstandard} $\fS_5\subset \fS_6$, which is linearizable.
An alternative proof of stable linearizability of $S$, using the equivariant torsor formalism, is in \cite[Proposition 20]{HT-stable}.
\end{exam}

\begin{exam}
\label{exam:a5}
We modify the previous example, considering $G=\fA_5$. Then 
$$
\wedge^2(W)=W_3\oplus W_3'\oplus W_4,
$$
and we put $V:=W_3\oplus W_4$. Then
$$
X:=\Gr(2,W)\cap \bP(V)
$$
is a smooth threefold \cite[Lemma 7.1.1]{CS}, the quintic Del Pezzo threefold. One of the main results of \cite{CS} is that $X$ is $G$-birationally rigid. 

Here, the construction of Proposition~\ref{prop:stable-rat} applies, and we obtain
$$
X\times \bP^1 \sim_G \bP(W). 
$$
Thus, $X\notin\Lin$ but $X\in \SLin$. 

To check the condition $(\ast )$ in Proposition~\ref{prop:stable-rat} and the smoothness of $X$ (independently of \cite{CS}), one can proceed as follows: we view $\mathfrak{A}_5$ as a subgroup of $\mathrm{PSL}_2$ and  
proceed in terms of $\mathrm{PSL}_2$-representations, as in \cite[Section 7]{CS}. Put $W=\mathrm{Sym}^4 (k^2)$ and 
consider the decomposition
\begin{equation} \label{eqn:split}
\wedge^2 (W) = \mathrm{Sym}^2 (k^2) \oplus \mathrm{Sym}^6 (k^2). 
\end{equation}
 Let $$\mathbf{H}=\begin{pmatrix} 1 & 0 \\ 0 & -1 \end{pmatrix}, \mathbf{X}=\begin{pmatrix} 0 & 1 \\ 0 & 0 \end{pmatrix}, \mathbf{Y}=\begin{pmatrix} 0 & 0 \\ 1 & 0 \end{pmatrix}$$ be the standard basis of the Lie algebra of $\mathrm{SL}_2$ satisfying
\[
[\mathbf{H}, \mathbf{X}]=2\mathbf{X}, \: [\mathbf{H}, \mathbf{Y}]=-2\mathbf{Y}, \: [\mathbf{X}, \mathbf{Y}]=\mathbf{H}.
\]
Let $w_4$ be a highest weight vector in $W$ (subscripts in the sequel indicate the weight). Thus $\mathbf{X}(w_4)=0$, and 
\[
w_4, \: w_2 =\mathbf{Y}(w_4), \: w_0 = \mathbf{Y}^2(w_4), \: w_{-2} = \mathbf{Y}^3(w_4), \: w_{-4} = \mathbf{Y}^4(w_4)
\]
form a basis for $W$; using the commutation relations inductively gives 
\[
\mathbf{X}(w_4)=0, \! \mathbf{X}(w_2)=4w_4, \! \mathbf{X}(w_0)=6w_2, \! \mathbf{X}(w_{-2})=6w_0, \! \mathbf{X}(w_{-4})=4w_{-2}. 
\]
 
To find a highest weight vector in the subrepresentation $\mathrm{Sym}^2 (k^2)$ of $\wedge^2 (W)$ one is looking for a linear combination of $w_4\wedge w_{-2}$ and $w_2\wedge w_0$ annihilated by $\mathbf{X}$. These are thus multiples of 
\[
x_2:=3 w_2\wedge w_0 - 2 w_4 \wedge w_{-2},
\]
and applying $\mathbf{Y}$ and $\mathbf{Y}^2$ to $x_2$ we obtain a basis for $\mathrm{Sym}^2 (k^2)$ as a submodule of $\wedge^2 (W)$ as
\begin{gather*}
x_0:=\mathbf{Y}(x_2)= \mathbf{Y}\left( 3 w_2\wedge w_0 - 2 w_4 \wedge w_{-2}\right) \\ = 3 (\mathbf{Y}(w_2)\wedge w_0 + w_2\wedge \mathbf{Y}(w_0)) - 2 (\mathbf{Y}(w_4)\wedge w_{-2} + w_4\wedge \mathbf{Y}(w_{-2}))\\
=3 (w_0\wedge w_0 + w_2\wedge w_{-2}) - 2 (w_{2}\wedge w_{-2} + w_4\wedge w_{-4})\\
= w_2\wedge w_{-2} - 2 w_4\wedge w_{-4}
\end{gather*}
and
\begin{gather*}
x_{-2}:= \mathbf{Y} \left( w_2\wedge w_{-2} - 2 w_4\wedge w_{-4} \right)\\
= w_0\wedge w_{-2} - w_2\wedge w_{-4} . 
\end{gather*}


Similarly, one can find a basis of $\mathrm{Sym}^6(k^2)\subset \wedge^2 (W)$ by applying $\mathbf{Y}$ successively to the highest weight vector $w_4\wedge w_2$ in that copy of $\mathrm{Sym}^6(k^2)$.  

We have thus explicitly identified both $\mathrm{Sym}^2 (k^2)$ and $\mathrm{Sym}^6 (k^2)$ with $\mathrm{PGL}_2$-subrepresentations of $\wedge^2 (W)$, and can check that $$X =\mathrm{Gr}(2, W) \cap \bP (\mathrm{Sym}^6 (k^2))$$ is irreducible, smooth of the expected dimension $3$ by computer algebra. The necessary checks were performed using {\ttfamily Macaulay2}\footnote{{\tt \tiny warwick.ac.uk/fac/sci/maths/people/staff/boehning/m2filesequivariantderived}}.
\end{exam}

\begin{exam}
\label{exam:4}
Let $G=C_9\rtimes C_6$, {\tt G:=SmallGroup(54,6)}.
Its smallest faithful representation has dimension 6, in particular, $G$ does not admit a linear action on $\bP^4$. Let $W$ be its unique irreducible representation of dimension 6, it has character
$$
( 6, 0, -3, 0, 0, 0, 0, 0, 0, 0 ). 
$$
We have a decomposition: 
\begin{equation}
    \label{eqn:wedge}
\wedge^2(W) = V_1\oplus V_1'\oplus V_1''\oplus V_2\oplus V_2'\oplus V_2''\oplus V_6
\end{equation}
into irreducible representations. 
Choose a suitable subrepresentation
$$
V:=V_1\oplus V_2\oplus V_2'\oplus V_6, 
$$
more precisely, that with respective characters, for $\zeta=\zeta_3$,

\begin{align*}
{\tt X.6}=& ( 1, -1, 1, \zeta^2, \zeta, -\zeta, -\zeta^2, \zeta, 1, 
\zeta^2),
\\
{\tt X.7}=&( 2, 0, 2, 2, 2, 0, 0, -1, -1, -1 ),
\\
{\tt X.8}=&( 2, 0, 2, 2\zeta, 2\zeta^2, 0, 0, -\zeta^2, -1, -\zeta),
\\
{\tt X.10}=&(6, 0, -3, 0, 0, 0, 0, 0, 0, 0 ).
\end{align*}
The complement decomposes as
\begin{align*}
{\tt X.2}=& (1, -1, 1, 1, 1, -1, -1, 1, 1, 1 ), \\
{\tt X.3}=& ( 1, -1, 1, \zeta, \zeta^2, -\zeta^2, -\zeta, \zeta^2, 1, \zeta),\\
{\tt X.9}=& ( 2, 0, 2, 2\zeta^2, 2\zeta, 0, 0, -\zeta, -1, -\zeta^2).
\end{align*}
Then, according to {\tt magma}, 
$$
X:=\Gr(2,W)\cap \bP(V)
$$
is a smooth and irreducible variety of dimension 4 and degree 14. Note that choosing a different 11-dimensional subrepresentation $V$ also yields irreducible fourfolds of degree 14, but some of these are singular. 
Thus the construction of Proposition~\ref{prop:stable-rat} applies, and we have
$X\notin \Lin$ and $X\in \SLin$.

Let $G=C_3^3\rtimes \fS_3$, {\tt SmallGroup(162,19)}. Its smallest faithful representation has dimension 6, in particular, $G$ does not admit a linear action on $\bP^4$. 
Let $W$ be an irreducible $G$-representation with character
$$
( 6, 0, -3, 0, 0, 3, -3, 0, 0, 0, 0, 0, 0 ). 
$$
We have a decomposition
$$
\wedge^2(W)=V_1\oplus V_2\oplus V_3\oplus V_3'\oplus V_6
$$
into irreducible representations. We choose 
$$
V:=V_2\oplus V_3\oplus V_6. 
$$
Then 
$$
X:=\Gr(2,W)\cap \bP(V)
$$
is irreducible (singular) of dimension 4 as can be checked by computer algebra\footnote{{\tt \tiny warwick.ac.uk/fac/sci/maths/people/staff/boehning/m2filesequivariantderived}}. Therefore this construction satisfies the hypotheses of Proposition \ref{prop:stable-rat}, thus $X\notin \Lin$ but $X\in \SLin$.

\end{exam}

\section{Derived categories}
\label{sect:derived}

\subsection*{Terminology}
Let $X$ be a smooth projective variety (over an algebraically closed field $k$ of characteristic zero) and $\rD^b (X)$ its derived category of coherent sheaves. We use freely the following terms; see, e.g., \cite[Section 2]{BBT23}, or \cite[Section 1 and 2]{kuz-der-view} for definitions and references: 
\begin{itemize}
\item admissible subcategories of $\rD^b (X)$,
\item exceptional objects, 
\item (full) exceptional sequences, 
\item (maximal) semiorthogonal decompositions.
\end{itemize}

\subsection*{$G$-actions on categories}

Let $G$ be an algebraic group, not necessarily finite. Let $X$ be a smooth projective $G$-variety, i.e., a smooth projective variety with a generically free, regular, action of $G$.  
In \cite[Proposition 3]{BBT23} it was remarked that the fundamental reconstruction theorem by Bondal and Orlov \cite{BO} admits the following equivariant version:

\begin{prop}
\label{prop:EquivRecon}
Suppose $X$ and $Y$ are smooth projective $G$-varieties over $k$, $X$ is Fano, and 
\[
\Phi\colon \rD^b (X) \simeq \rD^b (Y)
\]
is an equivalence as $k$-linear triangulated categories together with the induced $G$-actions. Then $X$ and $Y$ are isomorphic as $G$-varieties, i.e., there exists a $G$-equivariant isomorphism
\[
X \stackrel{\sim}{\lra} Y. 
\]
\end{prop}

In practice, this general theorem is not very useful since the derived category contains too much information; in the context of rationality problems, the focus is on trying to 
extract information about the variety from more accessible data, such as a piece, or several pieces, in a semiorthogonal decomposition of $\rD^b (X)$.

We will explore the extent to which these considerations
apply in the equivariant context. We investigate, in several representative geometric examples, the 
effects of $G$-equivariant birationalities on 
\begin{itemize}
\item 
the existence of full exceptional sequences in $\rD^b (X)$ that are compatible with $G$-actions, and 
\item 
derived Hom-spaces between objects in $\rD^b (X)$.
\end{itemize}



\subsection*{$G$-actions and exceptional sequences}

\begin{defi}\label{def:linearizedobject}
An object $E\in\rD^b (X)$ is called \emph{$G$-invariant} if $g^*E$ is isomorphic to $E$, for all $g\in G$. 
It is called 
\emph{$G$-linearized} if it is equipped with a $G$-linearization, i.e. a system of isomorphisms
\[
\lambda_g\colon E \to g^* E, \quad \forall g\in G,  
\]
satisfying the compatibility condition 
\[
\lambda_{1} = \mathrm{id}_{E}, \quad \lambda_{gh}= h^*(\lambda_g)\circ \lambda_h . 
\]
\end{defi}

Several notions of compatibility of exceptional sequences with $G$-actions have been studied; 
we follow \cite[Definition 2.1]{CaTe20}.

\begin{defi}\label{defi:EquivExcColl}
Let $X$ be a smooth projective $G$-variety and 
$$
{\bf E}:=(E_1, \dots , E_n)
$$ 
a full exceptional sequence in $\rD^b (X)$. 
\begin{enumerate}
\item
${\bf E}$ is \emph{$G$-invariant} if for every $r\in \{1, \dots , n\}$ and every $g\in G$, there is an $s$ such that $g^* E_r\simeq E_s$. 
\item
${\bf E}$ is $G$-\emph{equivariant} if it is 
$G$-invariant and, for all $r$, $E_r$ is isomorphic to a $G_r$-linearized object in $\rD^b (X)$, where $G_r\subseteq G$ is the stabilizer of the isomorphism class of $E_r$. 
\item 
${\bf E}$ is $G$-\emph{linearized} if 
it is $G$-equivariant and, for all $r$,  
$G_r=G$, i.e., each $E_r$ is a $G$-linearized object. 
\end{enumerate}
\end{defi}

\begin{exam}
\label{exam:EquivSequences}
Consider $X=\bP^1 \times \bP^1$, and 
the full exceptional sequence in $\rD^b (X)$ 
from \cite{Kap88} 
\[
{\bf E}= (\cO (-1,-1), \cO, \cO (1,0), \cO(0,1)).
\]
Then 
\begin{itemize}   
\item ${\bf E}$ is $H$-linearized, if we view $X$ as $\bP (V)\times \bP(V)$ with its natural diagonal $H$-action, where $H$ is a finite group admitting a two-dimensional faithful linear representation $V$ such that $H$ acts generically freely on $\bP (V)$.
\item  ${\bf E}$ is $G$-equivariant, but not $G$-linearized,  
if we let $G= \bZ/2 \times H$ act on $X= \bP (V)\times \bP(V)$, with the first factor $\bZ/2$ in $G$ switching the rulings,
\item ${\bf E}$ is $G$-invariant, but not $G$-equivariant, if 
we let $H \simeq \bZ/2\times \bZ/2$ act on $\bP^1$ via the two-dimensional faithful irreducible representation of its Schur cover $\mathfrak D_8$ instead, and then let $G= \bZ/2\times H$ act on $\bP^1 \times \bP^1$, again with $\bZ/2$ switching the factors and $H$ acting diagonally.
\end{itemize}
\end{exam}

\begin{exam}
\label{exam:CT}
The moduli space $\overline{\mathcal{M}}_{0,n}$ of stable rational curves with $n$ marked points has a full $\mathfrak{S}_n$-equivariant exceptional sequence, where $\mathfrak{S}_n$ is the symmetric group permuting the marked points, by the main result of \cite{CaTe20}. 
\end{exam}

The following observation will be useful in applications.

\begin{lemm}\label{lemm:InvEquiv}
Let $X$ be a smooth projective $G$-variety, and 
$$
{\bf E}:=(E_1, \dots , E_n)
$$ 
a $G$-invariant exceptional sequence consisting of line bundles. If $X$ is $G$-linearizable, then ${\bf E}$ is a $G$-equivariant exceptional sequence.
\end{lemm}

\begin{proof}
If $X$ is $G$-linearizable, $\mathrm{Am}(X, G)$ is trivial. The same holds for $\mathrm{Am}(X, H)$, for any $H\subseteq G$, since $X$ is also $H$-linearizable. Therefore, under the assumptions of the Lemma, any line bundle on $X$ is $H$-linearized for every subgroup $H$ that leaves this line bundle  invariant. 
\end{proof}

\subsection*{Connections with classical invariants}

\begin{prop}\label{pro:LinearizedExceptional}
Let $G$ be a finite group and $X$ a smooth projective $G$-variety admitting a $G$-linearized full exceptional sequence. Then 
$$
\mathrm{Pic}(X,G)\twoheadrightarrow \mathrm{Pic}(X)^G,
$$
in particular,
$$
\mathrm{Am}(X,G)=0.
$$
\end{prop}

\begin{proof}
Taking the first Chern class gives a well-defined homomorphism
\[
c_1\colon \rD^b (X) \to \mathrm{Pic}(X).
\]
If $\rD^b (X)$ is generated by an exceptional sequence ${\bf E}=(E_1, \dots , E_n)$ of $G$-linearized objects, then every class in $\mathrm{Pic}(X)$ is a $\bZ$-linear combination of the $c_1 (E_r)$, which are $G$-linearized. Indeed, the first Chern class of a $G$-linearized complex is $G$-linearized; it is the alternating sum of the Chern classes of the cohomology sheaves which are $G$-linearized, so we just need to show invariance of the Chern class  for a $G$-linearized sheaf. Such a sheaf always has a finite locally free resolution by $G$-linearized vector bundles since there exists a $G$-linearized ample sheaf on $X$ and the statement is true on projective space. 
\end{proof}

\begin{rema} 
\label{rema:conc}
Let $G$ be a finite group and $X$ a smooth projective $G$-variety with ample anticanonical class. By Proposition~\ref{prop:EquivRecon}, the derived category $\rD^b(X)$ determines $X$, as a $G$-variety. In particular, we can extract the $G$-action on $\Pic(X)$, and determine whether or not it satisfies 
{\bf (H1)} or {\bf (SP)}. Concretely, we have
$$
\Aut(\rD^b(X)) = (\Pic(X)\times \bZ)\rtimes \Aut(X),  
$$
and the derived automorphisms acting trivially on point objects can be identified with $\Pic(X)$. 
\end{rema}

In the following sections we investigate connections between existence of full exceptional sequences with various compatibility properties with the $G$-action, and (stable) linearizability of $X$. 
It turns out that $G$-linearizability often implies the existence of a full equivariant exceptional sequence, provided such sequences exist in the non-equivariant setting, as for Del Pezzo surfaces.

\section{Del Pezzo surfaces}
\label{sect:dp}

\subsection*{Terminology}
By the Minimal Model Program, every rational surface is birational to a conic bundle over $\bP^1$ or  
a Del Pezzo surface, i.e., a smooth projective surface $X$ with ample anticanonical class $-K_X$; we let 
$$
d=d(X)=(-K_X)^2
$$ 
be its degree.  The same holds over nonclosed field, and in presence of group actions. 

Here and below {\em conic bundle} means that $X$ is smooth and all fibers of 
$f \colon X \to \bP^1$ are isomorphic to reduced conics in $\bP^2$. We recall the terminology of 
\cite{DI, pro-II}: a conic bundle $f \colon X \to \bP^1$ is called {\em exceptional} if for some positive integer $g$ the number of degenerate fibers
equals $2g + 2$ and there are two disjoint sections $C_1$ and $C_2$ with $C_1^2=C_2^2 =-(g+1)$. 
Exceptional conic bundles can be constructed explicitly, see \cite[\S 5.2]{DI}.

\subsection*{Nonlinearizable actions}
In this section, $G$ is a finite group. 
The following nonlinearizability results
for $G$-conic bundles
are probably known to experts in birational rigidity; here, we rely on the Burnside formalism. 

\begin{lemm}\label{lemm:conic1}
Let $X\to \bP^1$ be a relatively minimal $G$-conic bundle with $K_X^2=1$. Then $X$ is not linearizable. 
\end{lemm}

\begin{proof}
If $X$ fails {\bf (H1)}, then 
then $X\notin\Lin$. If $X$ satisfies {\bf (H1)}, then the
classification in \cite[\S 8]{pro-II},  Theorem 8.3, shows that $G$ must be the binary dihedral group $\widetilde{\fD}_5$, a 
nontrivial central $C_2$-extension of $\fD_5$. Such $X$, with the $G$-action, are given by an explicit construction \cite[Construction 8.4]{pro-II}. In particular, there is a distinguished involution $\tau\in G$, generating the center of $G$ and fixing a smooth rational curve $C$. The residual action of $\fD_5=G/\langle \tau\rangle$ is generically free on $C$. Applying the Burnside formalism to this situation, we find a unique, incompressible, symbol
\begin{equation}
    \label{eqn:symb2}
(C_2, \fD_5\actsfromleft k(\bP^1), (1)), 
\end{equation}
contributing to the class 
$$
[X\actsfromright G] \in \Burn_2(G).
$$
A generically free linear action of $\widetilde{\fD}_5$ on $\bP^2$ 
necessarily arises from a representation
$V=V_1\oplus V_2$, where 
$V_1$ is 1-dimensional representation which is nontrivial on $\tau$ and 
$V_2$ is a 2-dimensional representation which is trivial on $\tau$. Then
$G$ fixes the point $p_0:=[1:0:0] \in \bP^2$ and stabilizes the line given by $x_0=0$. Passing to a standard model, we observe, as
in a similar situation in \cite[Section 7.6]{HKT-small}, that
the linear action contributes {\em two} symbols \eqref{eqn:symb2}, one 
from the exceptional divisor of the blowup of $p_0$ and the other from $\bP^1=\bP(V_2)$. 
It follows that 
$$
[X\actsfromright G ] \neq [\bP^2\actsfromright G]
$$
in $\Burn_2(G)$, and the $G$-action on $X$ is not linearizable. 
\end{proof}

\begin{lemm}\label{lemm:conic2}
Let $X$ be a minimal rational $G$-surface that is an exceptional conic bundle with $K_X^2=2$ and $g=2$. Then $X$ is not linearizable. 
\end{lemm}

\begin{proof}
If $X$ fails {\bf (H1)} 
then $X$ is not linearizable. 
The other cases have been classified in \cite[\S 8]{pro-II}: consider the representation
\[
\varrho\colon G \to \mathrm{Aut}(\mathrm{Pic}(X)),
\]
its kernel $\mathrm{ker} (\varrho )$, and  
the exact sequence
\begin{equation}
1 \to G_F \to G \to G_B \to 1,
\end{equation}
where $G_F\subset G$ is the largest subgroup acting trivially on the base $B=\bP^1$.
By \cite[Theorem 8.3]{pro-II}, we have $\mathrm{ker} (\varrho )\neq \{1\}$; by \cite[Theorem 8.6(2)]{pro-II} it is cyclic, whereas $G_B \simeq \mathfrak{D}_n$, with $n\ge 3$, or $ G_B\simeq \mathfrak{S}_4$. The table in \cite[Section 8.7]{pro-II} shows that $G_B=\fS_4$,  and  
\cite[Thm. 8.6]{pro-II} shows that $G_F =\mathrm{ker}(\varrho )=C_m$, a nontrivial cyclic group of order $m$.

Write $C_m =C_{2^r}\times C_{m'}$ with $\gcd(m',2)=1$, and consider a $2$-Sylow subgroup $G_2$ of $G$ that contains $C_{2^r}$. Then $G_2$ has order $2^r\times 8$ and sits in an extension
\[
1 \to C_{2^r} \to G_2 \to \bar{G}_2 \to 1
\]
where $\bar{G}_2$ is a subgroup of $\mathfrak{S}_4$ of order $8$, hence equal to $\mathfrak{D}_4$. Since the order of a group is divisible by the degree of any of its irreducible representations, every $3$-dimensional representations $V$ of $G_2$ has to decompose into irreducible summands of degrees $1,1,1$ or $1, 2$. Only the latter can be generically free. 
Thus, we may assume that $V$ is of the form 
\[
V=V_1\oplus V_2
\]
with $V_i$ irreducible of dimension $i$. Here $V_1=k_{\chi}$ is a representation of $G_2$ by some character $\chi$,  and we can assume that $V_1$ is trivial, and $V_2$ is a faithful $G_2$-representation. A standard model for $G_2$-action is the blowup 
$\widetilde{\bP(V)}\to \bP(V)$ of the $G_2$-fixed point 
$p_0=[1:0:0]$, see \cite[Section 7.2]{HKT-small}. The only incompressible divisorial symbols 
might arise from the exceptional divisor, respectively, the preimage of the projectivization $\bP^1=\bP(V_2)\subset \bP(V)$. The corresponding symbols are
$$
(C, G_2/C\actsfromleft k(\bP^1), (\chi)), \quad  (C, G_2/C\actsfromleft k(\bP^1), 
(\bar{\chi})), 
$$
where $C\subset G_2$ is a cyclic group and $\chi$ is a primitive character of $C$. 
Their sum in $\Burn_2(G_2)$ cannot equal to 
\[
(C_{2^r}, \mathfrak{K}_4 \actsfromleft k(\bP^1), (\psi)) 
\]
with $\psi$ some primitive character of $C_{2^r}$, for any choices of $C,\chi$. Thus  
$$
[X\actsfromright G_2 ] \neq [\bP^2\actsfromright G_2]
$$
in $\Burn_2(G_2)$, for any generically free linear action of $G_2$ on $\bP^2$.
\end{proof}

\subsection*{Linearization and derived categories}

We consider rational $G$-surfaces and investigate which pieces and properties of the $G$-category $\rD^b(X)$ are sensitive to geometric, and in particular, $G$-birational, characteristics of the $G$-action on $X$.


\begin{lemm}
\label{lem:InvSequenceH1}
Let $X$ be a rational $G$-surface $X$ admitting a full $G$-invariant exceptional sequence. Then the $G$-action on $\Pic(X)$ satisfies {\bf (H1)} and {\bf (SP)}. 
\end{lemm}

\begin{proof}
    The classes of the terms of the sequence in the Grothendieck $\rK$-group $\rK_0 (X)$ form a $\bZ$-basis that is permuted by $G$. Thus $\rK_0 (X)$ is a permutation module, and since
\[
\rK_0 (X) \simeq \bZ \oplus \mathrm{Pic}(X) \oplus \bZ
\]
as $G$-modules, with trivial $G$-action on the two summands $\bZ$, we obtain the claim. 
\end{proof}

Incidentally, assuming $X$ is a minimal $G$-Del Pezzo surface, Theorem 1.2 of \cite{pro-II} shows that {\bf (H1)} is equivalent to the fact that $G$ does not fix a curve of positive genus, and also equivalent to the condition $K_X^2\ge 5$ or $X$ being a special quartic Del Pezzo surface with a very special action, described in \cite[Thm. 1.2, (iii), (b)]{pro-II}.

\begin{lemm}\label{lem:NoEquivariant}
Let $X=\bP^2$ with a projectively linear but nonlinear action of a finite group $G$. Then $X$ does admit a full $G$-invariant exceptional sequence, but no full $G$-equivariant exceptional sequence.
\end{lemm}

\begin{proof}
The exceptional sequence $(\cO, \cO(1), \cO(2))$ is $G$-invariant. 
From \cite{Ku-Or94} it is known that every exceptional object in $\rD^b (X)$ is, up to shift, a vector bundle. Thus assume that $(\cE_1, \cE_2, \cE_3)$ is a $G$-equivariant full exceptional sequence consisting of vector bundles. Since the map
\[
\rD^b (X) \to \rK_0 (X)
\]
is $G$-equivariant and the action on $\rK_0(X)$ is trivial in this case, we see that every element in $G$ fixes the isomorphism class of each $\cE_i$ (because the images of $\cE_1, \cE_2, \cE_3$ form a $\bZ$-basis of $\rK_0 (X)$). If we compose $\rD^b (X) \to \rK_0 (X)$ with the first Chern class map, we get a surjective map to $\mathrm{Pic}(X)$. In other words, the top exterior powers of the $\cE_i$ generate the Picard group, hence at least one of them has to be isomorphic to $\cO_{\bP^2}(r)$ for some odd integer $r$. Thus $\cO_{\bP^2}(1)$ is also $G$-linearized, contradicting our assumption that the action is nonlinearizable.
\end{proof}

\subsection*{Examples with $\Aut(X)$-equivariant exceptional sequences}
We present examples of rational surfaces $X$
such that $\rD^b(X)$ admits a full $\mathrm{Aut}(X)$-equivariant exceptional sequence but
$X$ is not lineariable, 
for some $G\subseteq \Aut(X)$. 

\ 

\noindent
{\bf DP6:} 
Let $X$ be a Del Pezzo surface of degree $6$.
Then $X$ has full $\mathrm{Aut}(X)$-\emph{invariant}  exceptional sequence. 
Indeed, recall that there is an exact sequence
\[
0 \to T \to \mathrm{Aut}(X) \to W_X \to 0,
\]
where 
$$
W_X \simeq \bZ/2 \times \mathfrak{S}_3, \quad \mathrm{Aut}(X)  \simeq N(T) \rtimes \bZ /2,
$$ 
and $T$ is the maximal torus of $\PGL_3$, the quotient of $(k^\times)^3$ by the diagonal subgroup $k^\times$, $N(T)$ its normalizer. A generator of $\bZ/2$ in $W_X=\bZ/2 \times \mathfrak{S}_3$ can be identified with the lift of the standard Cremona involution on $\bP^2$ and $\mathfrak{S}_3$ is realized as the group of permutations of the points 
$$
p_1=(1:0:0),\quad p_2=(0:1:0), \quad p_3=(0:0:1)
$$ 
that are blown up to obtain $X$. 
There is always a full invariant exceptional collection for the entire automorphism group of $X$. Indeed, $X$ has the following (three block) exceptional sequence:
\begin{gather*}
\cO_X, \quad \cO_X (H), \quad \cO_X(2H- E_1-E_2-E_3), \\
\quad \cO_X (2H- E_1-E_2), \quad \cO_X (2H- E_2-E_3),\quad  \cO_X (2H- E_1-E_3),
\end{gather*}
where $H$ is the pullback of a hyperplane class and $E_i$ the exceptional divisors. The Cremona involution $\sigma$ acts as
\[
H \mapsto 2H -E_1-E_2-E_3, \quad E_i \mapsto H -E_j-E_k, \quad \{i,j,k\}=\{1,2,3\},
\]
whereas $\mathfrak{S}_3$ permutes the $E_i$ and fixes $H$, and $T$ fixes $H, E_i$. However, this sequence is not always an  equivariant exceptional sequence (for example, the normalizer of a maximal torus in $\PGL_3$ is in the stabilizer of $\cO_X (H)$, but the line bundle is not linearized since the action does not lift to a linear action on $k^3$).

\

\noindent
{\bf DP5:} Let $X$ be a Del Pezzo surface of degree $5$. We have 
$$
\Aut(X)\simeq \fS_5.
$$
By \cite[Theorem 1.2 and Example 1.3]{CaTe20}, $X$ has a full $\mathfrak{S}_5$-equivariant exceptional collection. However, $X$ is $G$-superrigid for $G=\mathfrak{A}_5$ \cite{DI}. 

\

\subsection*{A special DP4}

By \cite[Theorem 1.2]{pro-II}, there is a unique minimal $G$-Del Pezzo surface $X$ of degree $\le 4$ that satisfies {\bf (H1)}. 
It is a Del Pezzo surface of degree 4, an interection of two quadrics in $\bP^4$
\begin{equation}
\label{eqn:X}
x_1^2+ \zeta x_2^2 + \zeta^2 x_3^2 + x_4^2 = x_1^2+ \zeta^2 x_2^2+ \zeta x_3^2 +x_5^2 =0,
\end{equation}
with $\zeta=\zeta_3$ a primitive cube root of unity, and $G=\bZ/3\rtimes \bZ/4$, 
with generators
\begin{align*}
\gamma\colon (x_1, x_2, x_3, x_4, x_5) &\mapsto (x_2, x_3, x_1,\zeta x_4, \zeta^2 x_5), \\
\beta'\colon (x_1, x_2, x_3, x_4, x_5)  &\mapsto (x_1, x_3, x_2, -x_5, x_4).
\end{align*}

\begin{theo}
\label{theo:Degree4Prokh}
The derived category $\rD^b(X)$ of the minimal $G$-Del Pezzo surface $X$ given by \eqref{eqn:X}
does not admit a full $G$-invariant exceptional sequence. 
\end{theo}

\begin{proof}
Arguing by contradiction, we assume that such a sequence 
\[
(\mathcal{E}_1, \dots , \mathcal{E}_8)
\]
exists. The group $G$ acts on the terms of the sequence by permutations, decomposing 
the set of terms into $G$-orbits, each of which is again an exceptional sequence. 
Let 
\[
(\mathcal{F}_1, \dots , \mathcal{F}_r)
\]
be one of the orbits. Consider the classes $v_i$ of the $\mathcal{F}_i$ in the Grothendieck group
\[
\rK_0(X)\simeq \bZ \oplus \mathrm{Pic}(X) \oplus \bZ \simeq \bZ^8. 
\]
Let $\chi (-, -)$ be the Euler bilinear pairing on $\rK_0(X)$ and $v:=v_r= [\mathcal{F}_r]$. Since $(v_1, \dots , v_r)$ is a numerically exceptional sequence with respect to the Euler pairing, we have
\[
 (\ast)\quad\quad \chi \bigl(v, g(v)\bigr)= \left\{ 
 \begin{matrix}
     0 & \text{if $g(v)\neq v$,}\\
     1 & \text{if $g(v)=v$,}
 \end{matrix}\right. 
\]
for all $g\in G$. 
These are quadratic equations for the coefficients of $v$. 
Let $H\subseteq G$ be the subgroup fixing $v$. If $H=G$, then $r=1$ and $v_r=v_1$ is $G$-invariant. If $H=1$, then $r=12$, a contradiction. Let us now assume that $H$ is a nontrivial proper subgroup of $G$. There are six such subgroups and they are all cyclic. Let $\rK_0(X)^H\simeq \bZ^r$ be the space of $H$-invariants and consider the ideal $I_H$ generated by the conditions $(\ast)$ in $\bZ[s_1, \dots , s_r]$. One can show 
that $I_H= (1)$ mod $3$ for all such $H$, e.g., with {\ttfamily Macaulay2}\footnote{{\tt \tiny warwick.ac.uk/fac/sci/maths/people/staff/boehning/m2filesequivariantderived}}. 
Hence only $H=G$ is possible. Since this is the case for all $G$-orbits, we obtain that all classes $[E_i]\in \rK_0(X)$ are $G$-invariant. However, they also form a $\bZ$-basis of $\rK_0(X)$. But $\rK_0(X)^G \neq \rK_0(X)$.

More precisely,  we proceed as follows. Since $X$ is a DP4, there are $16$ lines on $X$. To determine these lines explicitly consider the rank $2$ skew matrix
\[
L = {\tiny \begin{pmatrix}
     0&    1&    -1&      1&      1 \\
    -1&    0&     1&   \zeta^2&     \zeta \\
     1&   -1&     0&     \zeta&   \zeta^2 \\
    -1&-\zeta^2&   -\zeta&      0&\zeta-\zeta^2 \\
    -1&  -\zeta&-\zeta^2&-\zeta+\zeta^2&      0
\end{pmatrix}}\in \mathrm{Gr}(2, W) \subset \bP (\wedge^2(W)), 
\]
with $\dim(W)=5,$
and a diagonal matrix $D$ with entries $\pm1$ on the diagonal. Then  we check that $DLD^t$ represents a line on $X$. This gives the $16$ lines on $X$ which are permuted by $\beta$ and $\gamma$. Observe that the line represented by $L$ is $\gamma$-invariant. 

We now choose $6$ lines whose classes are a basis of $\Pic(X)$ as follows. There are precisely $5$ lines $L_1, \dots , L_5$ that intersect the line represented by $L$. Two of these lines are $\gamma$-invariant. Without loss of generality, we can assume these are $L_1$ and $L_2$. Finally there is a unique fourth $\gamma$-invariant line $L_6$. 

Then $L_1, \dots , L_6$ form a basis of $\Pic(X)$: indeed their intersection matrix can be computed as
\[
\left(
\begin{smallmatrix}
      {-1}&0&0&0&0&1\\
      0&{-1}&0&0&0&1\\
      0&0&{-1}&0&0&0\\
      0&0&0&{-1}&0&0\\
      0&0&0&0&{-1}&0\\
      1&1&0&0&0&{-1}\end{smallmatrix}
\right)
\]
We can now compute the representation of the other lines in this basis  by considering their intersections with the lines in the given basis. We find representations
{\tiny
\[
    B = \left(\begin{matrix}
      1&1&{-1}&{-1}&{-1}&2\\
      0&0&0&0&0&1\\
      1&0&0&{-1}&0&1\\
      1&0&{-1}&0&0&1\\
      1&0&0&0&{-1}&1\\
      1&0&0&0&0&0\end{matrix}\right)
      \quad \text{and} \quad
    C = \left(\begin{matrix}
      1&0&0&0&0&0\\
      0&1&0&0&0&0\\
      0&0&0&0&1&0\\
      0&0&1&0&0&0\\
      0&0&0&1&0&0\\
      0&0&0&0&0&1\end{matrix}\right)
\]
}

\noindent
of $\beta$ and $\gamma$, respectively. 
Moreover, the canonical class $K_X$ is determined by the fact that the intersection number of $-K_X$ with all lines is equal to $1$. One finds:
\[
-K_X= 2L_1 +2L_2-L_3-L_4 -L_5 +3L_6,
\]
which is also equal to the sum of the four $\gamma$-invariant lines $L, L_1, L_2, L_6$.

\medskip

Following \cite[Section 3]{BB}, we work out the Euler pairing explicitly. The Chern character 
\begin{align*}
    \mathrm{ch}\colon \mathrm{K}_0(X) & \to \mathrm{CH}^* (X)_{\bQ} \\
    [\cE] & \mapsto \mathrm{rk}(E) + c_1 (\cE ) + \frac{c_1(\cE)^2 -2c_2 (\cE)}{2} 
\end{align*}
is an injective ring homomorphism with values in the sublattice 
\[
\Lambda:= \left\{ x   + y_1  l_1 + y_2 l_2 + \dots + y_6 l_6 + \frac{1}{2} z p \right\}\simeq \bZ^8 \subset \mathrm{CH}^* (X)_{\bQ},
\]
where $(x, y_1, y_2, \dots , y_6, z)\in \bZ^{8}$, $p$ is the class of a point, and $l_i:=c_1(\cO(L_i))$. We set $v= (x,y,z)$, where
\[
y = y_1  l_1 + y_2 l_2 + \dots + y_6 l_6.
\]
Thus $\Lambda$ is generated by $\mathrm{CH}^0(X)\simeq \bZ$, $\mathrm{CH}^1(X)\simeq \mathrm{Pic}(X) \simeq \bZ^6$ and $\frac{1}{2}\mathrm{CH}^2(X)$, where $\mathrm{CH}^2(X)\simeq \bZ$ is generated by the Chern character of the skyscraper sheaf of a point $p$, which is just the class of $p$ in the Chow ring. Its image is an index $2$ sublattice $\mathrm{ch}(\mathrm{K}_0(X)) \subset \Lambda$: indeed, $\frac{1}{2}p$ is not in $\mathrm{ch}(\mathrm{K}_0(X))$ since for the Euler pairing $\chi$ 
\[
\chi (\cO_X, \cO_p) =1
\]
and $\chi$ takes integral values on $\mathrm{ch}(\mathrm{K}_0(X))$. The class of $\frac{1}{2}p$ generates the quotient $\Lambda/ \mathrm{ch}(\mathrm{K}_0(X))$. 
By Riemann-Roch,
\[
\chi (X, \mathcal{E}) = \deg \left( \mathrm{ch}(\mathcal{E}).\mathrm{td}(\mathcal{T}_X)\right)_2,
\]
where
\[
\mathrm{td}(\mathcal{T}_X) = 1 - \frac{1}{2}K_X + \frac{1}{12} (K_X^2 + c_2 ) = 1 - \frac{1}{2} K_X + p.
\]
The subscript $2$ in the second to last formula means that one only considers the top-dimensional component. Hence in terms of $v= (x,y,z)$,
\[
\chi (X, \mathcal{E}) = x - \frac{1}{2} y . K_X + \frac{1}{2} z \, . 
\]
If $\mathcal{E}_1$ and $\mathcal{E}_2$ are bundles, then 
\[
\chi (\mathcal{E}_1, \mathcal{E}_2 ) = \chi (X, \mathcal{E}_1^{\vee} \otimes \mathcal{E}_2 )
\]
and 
\begin{align*}
\mathrm{ch}(\mathcal{E}_1^{\vee} \otimes \mathcal{E}_2 ) &= \mathrm{ch}(\mathcal{E}_1^{\vee}). \mathrm{ch} (\mathcal{E}_2) = (x_1 - y_1 + \frac{1}{2} z_1 )( x_2 + y_2 + \frac{1}{2} z_2) \\
&= x_1x_2  + (x_1y_2 - x_2 y_1) + \frac{1}{2}(x_1z_2 + x_2z_1 - 2y_1y_2),
\end{align*}
whence
\[
\chi (\mathcal{E}_1, \mathcal{E}_2 ) = x_1x_2 - \frac{1}{2} (x_1y_2 - x_2 y_1).K_X + \frac{1}{2} (x_1z_2 + x_2z_1-2y_1y_2) \, .
\]

We work out the Euler pairing $\chi$ on the lattice $\Lambda$ in the above $\bZ$-basis: 
\[
\left(
\begin{smallmatrix}
      1&\frac{1}{2}&\frac{1}{2}&\frac{1}{2}&\frac{1}{2}&\frac{1}{2}&\frac{1}{2}&\frac{1}{2}\\
      {-\frac{1}{2}}&1&0&0&0&0&{-1}&0\\
      {-\frac{1}{2}}&0&1&0&0&0&{-1}&0\\
      {-\frac{1}{2}}&0&0&1&0&0&0&0\\
      {-\frac{1}{2}}&0&0&0&1&0&0&0\\
      {-\frac{1}{2}}&0&0&0&0&1&0&0\\
      {-\frac{1}{2}}&{-1}&{-1}&0&0&0&1&0\\
      \frac{1}{2}&0&0&0&0&0&0&0\end{smallmatrix}
      \right)
\]

We now show that equations $(\ast)$ cannot be solved even in the larger lattice $\Lambda$. Namely, consider the subgroup $H\subset G$ generated by $\beta$. The invariants of 
$\beta$ in $\Lambda$ are
\[
     v = ({z}_{1},{-2\,{z}_{3}},{-2\,{z}_{3}},-{z}_{2}-{z}_{3
       },{z}_{2}+3\,{z}_{3},{z}_{3},{-3\,{z}_{3}},{z}_{4})
\]
for $z_i \in \bZ$. Now
\begin{align*}
    1 & = \chi\bigl(v,v\bigr) = {z}_{1}^{2}+2\,{z}_{2}^{2}+8\,{z}_{2}{z}_{3}+4\,{z
       }_{3}^{2}+{z}_{1}{z}_{4}, \\
    0 & = \chi\bigl(v,vC\bigr) = {z}_{1}^{2}-{z}_{2}^{2}-4\,{z}_{2}{z}_{3}-8\,{z}_{3
       }^{2}+{z}_{1}{z}_{4}.
\end{align*}
Subtracting the second equation from the first we obtain
\[
    1 =  3\,{z}_{2}^{2}+12\,{z}_{2}{z}_{3}+12\,{z}_{3}^{2}
\]
which has no solution modulo $3$.

The same computation can be done for all nontrivial proper subgroups of $G$.
\end{proof}

\subsection*{Implications}

Figure 1 shows relations between the  different notions for a minimal $G$-Del Pezzo surface $X$ with $\mathrm{rk}\, \mathrm{Pic}(X)^G =1$. 

\begin{figure}[H]\label{f:Implications}
\xycenter{
*+[F]\txt{$X$ has a full\\ exceptional sequence \\ of $G$-linearized objects}\ar@{=>}[d]^-{\circled{4}} \ar@{<=>}[r]^-{\circled{5}} &
*+[F]\txt{$X=\bP^2\in \Lin$} \\
*+[F]\txt{$X\in \Lin$}\ar@{=>}[d]^{\circled{1}} &  \\
*+[F]\txt{$X$ has a full equivariant exceptional sequence}\ar@{=>}[d]^{\circled{2}} & \\
*+[F]\txt{$X$ has a full invariant exceptional sequence}\ar@{=>}[d]^{\circled{3}} & \\
*+[F]\txt{$X$ satisfies {\bf (H1)}}  & 
}
\caption{}
\end{figure}

\begin{itemize}
\item
The implications \circled{1}-\circled{4} are strict, see below for references to proofs. 
\item 
\circled{3} is proven in Lemma \ref{lem:InvSequenceH1}, whereas \circled{2}, \circled{4} are immediate from the definitions once \circled{5} is proven. 
\item 
\circled{1} is not reversible, e.g., for $X$ a DP6. 
\item 
\circled{5} follows from  Proposition~\ref{pro:LinearizedExceptional} since $X$ then has Picard rank $1$, and this also shows that \circled{4} is not reversible. 
\item \circled{2} is not reversible, by Lemma~\ref{lem:NoEquivariant}. 
\item \circled{3} is not reversible, by Theorem~\ref{theo:Degree4Prokh}.  
\end{itemize}

The main result, which requires a longer argument, is the implication \circled{1}. We will prove this more generally whenever $X$ is a smooth rational $G$-surface. 

\begin{theo}
\label{theo:DelPezzo}
A smooth projective rational $G$-surface that is linearizable has a full $G$-equivariant exceptional sequence. 
\end{theo}

The proof will occupy the remainder of this section. It is based on a detailed analysis of actions, following \cite{DI} and \cite{pro-II}.

\begin{proof}
We assume that $X$ is linearizable. 

\

\noindent
\textbf{Step 1.} 
We reduce to $G$-minimal surfaces: 
Indeed, consider a blowup $\tilde{X}\to X$ in a $G$-invariant set of points.
By Orlov's blowup formula \cite{orlov-mon}, if $X$ admits a full $G$-equivariant exceptional sequence, then so does $\tilde{X}$. 
The stabilizer $G_x\subseteq G$ of a point $x\in X$ acts linearly on the tangent bundle of $X$ at $x$; hence the $G_x$-action on the sheaves $\cO_E (r)$ (where $E$ is the exceptional divisor over $x$) is linearized.

\

\noindent
\textbf{Step 2.} 
By \cite[Thm. 3.8]{DI}, a minimal rational $G$-surface $X$ either admits a structure of a $G$-conic bundle over $\bP^1$ with $\mathrm{Pic}(X)^G \simeq \bZ^2$ or $X$ is isomorphic to a Del Pezzo surface with $\mathrm{Pic}(X)^G \simeq \bZ$.
We proceed via classification in \cite[Section 8]{DI}, depending on the possible values of $d=K_X^2$.

\

\noindent
\textbf{Step 3.}

\noindent
\textbf{Case $d\le 0$:} 
$X$ is a rigid $G$-conic bundle with $8-d$ singular fibres, and in particular, $X\notin\Lin$. 

\medskip

\noindent
\textbf{Case $d=1$:} 
$X$ is a rigid $G$-Del Pezzo surface, thus $X\notin \Lin$, or a $G$-conic bundle, treated in Lemma~\ref{lemm:conic1}. 
\medskip

\noindent
\textbf{Case $d= 2$:}  $X$ is a 
rigid $G$-Del Pezzo surface, thus $X\notin \Lin$, or a $G$-conic bundle. 
If the conic bundle is not exceptional, it is rigid; exceptional conic bundles with $g=2$ are treated in 
Lemma~\ref{lemm:conic2}. 

\medskip

\noindent
\textbf{Case $d= 3$:}  
$X$ is either a minimal $G$-Del Pezzo surface that is rigid, thus $X\notin \Lin$; or a minimal $G$-conic bundle, in which case $G$ contains three commuting involutions two of which have fixed point curves of genus $2$, yielding the {\bf (H1)}-obstruction to linearizability, contradicting the assumption.  

\medskip

\noindent
\textbf{Case $d= 4$:} 
$X$ can be a minimal $G$-Del Pezzo surface. If $X^G=\emptyset$, $X$ is either rigid or superrigid, hence $X\notin \Lin$. If $X^G \neq\emptyset$, then $X$ is $G$-birational to a minimal conic bundle with $d=3$ and we conclude as in the previous case.

If $X$ is a minimal $G$-conic bundle, then either $X$ is an exceptional conic bundle with $g = 1$: assuming that $X$ is linearizable, \cite[Theorem 8.3]{pro-II} implies that the kernel of 
\[
\varrho\colon G \to \mathrm{Aut}(\mathrm{Pic}(X))
\]
is non-trivial, since otherwise $K_X^2$ has to be odd. Then \cite[Classification in \S 8.1]{DI} implies that no elementary transformation is possible and $X$ is not $G$-birational to any Del Pezzo surface, hence $X\notin \Lin$.

Secondly, $X$ can also be a $G$-Del Pezzo surface with two sections with self-intersection
$-1$ intersecting at one point. In this case, $X$ is obtained by regularizing a de Jonqui\`{e}res involution; since such a de Jonqui\`{e}res involution is not conjugate to a projective involution, $X\notin \Lin$. 

\medskip

\noindent
\textbf{Case $d= 5$:} has been considered above.  

\medskip

\noindent
\textbf{Case $d= 6$:} 
$X$ always has a full $G$-invariant exceptional sequence. If $X\in \Lin$, Lemma \ref{lemm:InvEquiv} applies. 

\medskip

\noindent
\textbf{Case $d= 8$:}  
If $X=\mathbb{F}_0=\bP^1\times \bP^1$, then it has the full exceptional sequence, see \cite{Kap88}, 
\[
{\bf E}= (\cO (-1,-1), \cO, \cO (1,0), \cO(0,1)),
\]
which is invariant under the full automorphism group $$\mathrm{Aut}(X)=\mathrm{PGL}_2 (\bC ) \wr  C_2.$$  
If $X$ is linearizable for a subgroup $G\subset \Aut(X)$, then Lemma \ref{lemm:InvEquiv} applies. In that case, every $G$-invariant full exceptional sequence is a $G$-equivariant full exceptional sequence.

When $X=\bF_n$ with $n\ge 2$, we apply Proposition \ref{prop:HirzebruchSurfaces}.

\medskip

\noindent
\textbf{Case $d= 9$:}  
$X=\bP^2$, and there is nothing to show. 
\end{proof}

\begin{prop}\label{prop:HirzebruchSurfaces}
Let $X$ be a $G$-Hirzebruch surface $\mathbb{F}_n$, $n\ge 2$, that is $G$-linearizable. Then $X$ admits a full $G$-equivariant exceptional sequence. 
\end{prop}

\begin{proof}
If $X =\bF_n$, $n\ge 2$, \cite[Theorem 4.10]{DI} shows that any finite subgroup $G \subset \mathrm{Aut}(X)$ is contained in $\mathrm{GL}_2 (k)/\mu_n$, which is embedded into $\mathrm{Aut}(X)$ as follows: view $\bF_n$ as the quotient $(\bA^2\backslash\{0\})^2/\mathbb{G}_m^2$, acting by
\begin{align*}
    \mathbb{G}_m^2 \times (\bA^2\backslash\{0\})^2 & \to (\bA^2\backslash\{0\})^2\\
    \bigl((\lambda, \mu), (x_0,x_1,y_0, y_1)\bigr) & \mapsto (\lambda\mu^{-n} x_0, \lambda x_1, \mu y_0, \mu y_1),  
\end{align*}
and with projection 
\begin{align*} 
\pi \colon \mathbb{F}_n & \to \bP^1\\
(x_0,x_1,y_0, y_1) & \mapsto (y_0:y_1),
\end{align*}
identifying 
$$
\bF_n=\mathbb{P}(\cO_{\bP^1}(n)\oplus \cO_{\bP^1}),
$$
as a $\bP^1$-bundle. Letting $A=(a_{ij}) \in \mathrm{GL}_2(k)$ act on the $y$-coordinates
\[
A\cdot (x_0,x_1,y_0, y_1) = (x_0,x_1,a_{11}y_0 +a_{12}y_1, a_{21}y_0+a_{22}y_1),
\]
we obtain an action of $\mathrm{GL}_2(k)$ on $\bF_n$; clearly $\mu_n\subset \mathrm{GL}_2(k)$ acts trivially on $\bF_n$, and we get an induced action of $\mathrm{GL}_2 (k)/\mu_n$. Actually the full automorphism group of $\bF_n$ is a semidirect product of $\mathrm{GL}_2 (k)/\mu_n$ by a normal subgroup $k^{n+1}$, thought of as the space of binary forms of degree $n$ with its natural action of $\mathrm{GL}_2 (k)/\mu_n$, because $\bF_n$ can also be realized as the blowup of the weighted projective space $\bP (1,1,n)$ at its singular point.

\medskip

The group $\mathrm{GL}_2 (k)/\mu_n$ is a central product of $k^\times$, embedded diagonally, and  $\mathrm{SL}_2(k)$, intersecting in the subgroup generated by $-\mathrm{id}$, for $n$ odd. Then every term in the exceptional sequence (using the relative version of Beilinson's theorem as in \cite{orlov-mon})
\[
(\ast) \quad \quad (\cO_{\bP^1}, \cO_{\bP^1}(1), \pi^*(\cO_{\bP^1}), \pi^*(\cO_{\bP^1}(1))\otimes \cO_{\bP(\mathcal{E})}(1)),
\]
where $\mathcal{E}=\cO_{\bP^1}(n)\oplus \cO_{\bP^1}$ and $\cO_{\bP(\mathcal{E})}(1)$ is the relative hyperplane bundle on $\mathbb{P}(\cO_{\bP^1}(n)\oplus \cO_{\bP^1}) \to \bP^1$, is invariant under $\mathrm{GL}_2 (k)/\mu_n$. We conclude by Lemma \ref{lemm:InvEquiv}.  
 \end{proof}

\section{Threefolds}
\label{sect:three}

There is a wealth of results concerning linearizability of $G$-actions on rational threefolds, in the context of birational rigidity. 
In absense of this property, only few examples are known. Of particular interest are threefolds without obvious obstructions, such as nontriviality of the Amitsur invariant or failure of {\bf (H1)}. 

In the arithmetic context, rationality over nonclosed fields of smooth geometrically rational Fano threefolds, e.g., those of Picard number one, has been investigated in  \cite{HTQuadrics21}, \cite{HT-Fano}, \cite{BW},
\cite{Kuz-Pro}:
\begin{itemize}
    \item[(1)] $V_5$,
    \item[(2)] $\bP^3$, $Q_3$,  $X_{12}$, $X_{22}$,
    \item[(3)] $V_4$, $X_{16}$, $X_{18}$ 
\end{itemize}
(we use standard notation for the threefolds from those papers). 
As is well-known, $V_5$, a Del Pezzo threefold of degree 5, is always rational. The rationality of forms of varieties in group (2) is controlled by the existence of rational points. In addition to this condition, rationality of varieties in group (3), i.e., forms of complete intersections of two quadrics, Fano threefolds of degree 16, 18,  
requires the existence, over the ground field, 
of lines, twisted cubics, or conics, respectively.   
The papers \cite{Kuz-Pro}, \cite{K-families} put this into the framework of derived categories, investigating the semiorthogonal decompositions in this context. 
In particular, the only cases where the semiorthogonal decomposition does not involve 
Brauer classes from the base, are 
$V_5$ and $X_{12}$, by \cite[Theorem 1.1 and Theorem 1.3]{K-families}.

We turn to the equivariant setting and linearizability questions.  
The case of quadrics is already involved \cite[Section 9]{TYZ}:
\begin{itemize} 
\item 
existence of fixed points is not necessary for linearizability, when $G$ is nonabelian,
\item 
there are cases, when linearizability is obstructed by the Burnside formalism, but stable linearizability is open, 
\item 
there are cases with no visible obstructions, but resistant to all attempts to linearize the action. 
\end{itemize}
As we are interested in situations 
where no obstructions are visible in the derived category, we focus on $V_5$ and $X_{12}$. 

\subsection*{Quintic Del Pezzo threefolds}

As in Example~\ref{exam:a5}, let $W$ be a faithful 5-dimensional representation of  a finite group $G$, such that $\wedge^2(W)$ contains a faithful 7-dimensional subrepresentation $V$, so that 
$$
X=\Gr(2,W)\cap \bP(V),
$$
is a smooth threefold, with generically free action of $G$, a quintic Del Pezzo threefold. 
The restriction $\cU_X$ of the universal rank-2 subbundle $\cU$ over $\Gr(2,W)$ to $X$ 
is naturally $G$-linearized. 
Orlov \cite{OrlV5} showed that there is a full exceptional sequence in $\rD^b (X)$ of the form 
\[
\langle (W \otimes \cO)/\mathcal{U} \otimes \cO(-1), \mathcal{U}, \cO, \cO(1)  \rangle . 
\]
This is a sequence of $G$-linearized objects.

There is a distinguished such threefold, with $G=\fA_5$-action, considered in Example~\ref{exam:a5}. It is:  
\begin{itemize}
    \item $G$-birationally rigid, and thus $X\notin \PLin$, and  
    \item $X\in \SLin$. 
\end{itemize}
On the other hand, recall that there is a longstanding conjecture in the context of 
derived categories relating the rationality of
a smooth projective variety (over an algebraically closed field of charactertistic zero) 
to the existence of a full exceptional sequence in $\rD^b(X)$. 
The above example contradicts the most suggestive analog of this conjecture in the equivariant context. Note that in the arithmetic context, every form of a quintic Del Pezzo threefold is rational, see, e.g., \cite[Theorem 1.1]{Kuz-Pro}.

\subsection*{Fano threefolds of genus 7}
We follow the discussion in \cite{Mukai92, Mukai95} and its summary in \cite{Kuz05} and \cite{AG5}. Consider a ten-dimensional complex vector space $V$ with a non-degenerate symmetric bilinear form on it, and denote by $\mathrm{Spin}_{10}$ the associated spinor group with $16$-dimensional half-spinor representations $\mathrm{S}^\pm V$. Consider the Lagrangian Grassmannian of $5$-dimensional isotropic subspaces of $V$: it has two connected components $\mathrm{LGr}_+ (V)$ and $\mathrm{LGr}_{-}(V)$ that can be identified with the closed orbits of the group $\mathrm{Spin}_{10}$ in $\bP (\mathrm{S}^+ V)$ and $\bP (\mathrm{S}^- V)$. The representations $\mathrm{S}^+ $ and $\mathrm{S}^- $ are dual to each other. Choose a pair of subspaces and their orthogonal subspaces (subscripts denote dimensions)
\[
A_8 \subset A_9 \subset \mathrm{S}^+ V, \quad B_7 \subset B_8 \subset \mathrm{S}^- V .
\]
Let
\begin{align*}
    X := & \mathrm{LGr}_+ (V) \cap \bP (A_9) \subset \bP (\mathrm{S}^+ V), \\
    S : =& \mathrm{LGr}_+ (V) \cap \bP (A_8) \subset \bP (\mathrm{S}^+ V)
\end{align*}
and
\begin{align*}
    C^{\vee} := & \mathrm{LGr}_{-} (V) \cap \bP (B_7) \subset \bP (\mathrm{S}^- V), \\
    S^{\vee} : =& \mathrm{LGr}_{-} (V) \cap \bP (B_8) \subset \bP (\mathrm{S}^- V). 
\end{align*}
It is known that $X$ is smooth if and only if $C^{\vee}$ is smooth, and $S$ is smooth if and only if $S^{\vee}$ is smooth, which we will now assume. Then $X=X_{12}$ is an index $1$ degree $12$ genus $7$ Fano threefold with a smooth K3 hyperplane section $S$ (a polarized K3 surface of degree $12$); all such pairs $(X, S)$ are obtained via the above linear algebra construction by \cite{Mukai92, Mukai95}. Moreover, $S^{\vee}$ is also a K3 surface of degree $12$ and $C^{\vee}$ is a canonically embedded curve of genus $7$. 
Note that restricting the universal bundle $\mathcal{U}$ from $\mathrm{Gr}(5, V)$, we obtain rank $5$ vector bundles $\mathcal{U}_+, \mathcal{U}_{-}$ on $\mathrm{LGr}_{+} (V)$ and $\mathrm{LGr}_{-} (V)$. 

It is known that the Sarkisov link with center a general point $x\in X$ gives a birational map to a quintic Del Pezzo threefold, hence $X$ is rational \cite[Theorem 5.17 (i)]{Kuz-Pro}, \cite{AG5}. 

We now pivot to the equivariant setup, following \cite[Example 2.11]{prokh3}. Consider $G=\mathrm{SL}_2(\mathbb{F}_8)$ and let $U$ 
be a 9-dimensional irreducible representation of $G$. 
There is a unique $G$-invariant quadric $Q\subset \bP (U)$, with generically free $G$-action.  
Note that, quite generally, the spinor varieties for $\mathrm{Spin}_{2n-1}$ and $\mathrm{Spin}_{2n}$ are isomorphic, indeed, projectively equivalent as subvarieties of projective space $\bP^N$, $N=2^{n-1}-1$, in their spinor embeddings. So we can also think of $\mathrm{LGr}_+ (V)$ as well as $\mathrm{LGr}_- (V)$ as the spinor variety for $\mathrm{Spin}_9$, parametrizing projective spaces of dimension $3$ on a smooth quadric in $\bP^8$. Hence $G$ acts on $\mathrm{LGr}_+ (V)$ (and $\mathrm{LGr}_- (V)$). The embedding of these into $\bP (S^+ V)$ and $\bP (S^- V)$ is given by the positive generator of the Picard group, eight times of which is the anticanonical bundle. According to \cite[Example 2.11]{prokh3}, the group $G$ acts in $\bP^{15}$ with invariant projective subspaces of dimensions $8$ and $6$ such that we get an action of $G$ on $X$ and $C^{\vee}$.

The group $G$ contains the Frobenius group $\mathfrak F_8$, which  
does not act on $\bP^3$, so that the $G$-action on $X$ is not linearizable. 

\medskip

The derived category of $X$ is described in \cite[Theorem 4.4]{Kuz05} and \cite[Theorem 5.15]{K-families}. It has a semiorthogonal decomposition 
\[
\rD^b (X) = \langle \mathcal{U}_+, \cO_X, \rD^b (C^{\vee}) \rangle . 
\]
The proof uses the interpretation of $C^{\vee}$ as the moduli space of stable rank $2$ vector bundles on $X$ with $c_1=1, c_2=5$ given in \cite{IM04}. 
Indeed, define $$
\mathcal{A}_X:=  \, ^\perp\langle \mathcal{U}_+, \cO_X \rangle.
$$
Kuznetsov constructs a fully faithful Fourier-Mukai functor 
\[
\Phi_{\mathcal{E}} \colon \rD^b (C^{\vee}) \to \mathcal{A}_X \subset \rD^b (X)
\]
from the universal bundle $\mathcal{E}$ on $X\times C^{\vee}$, and this is an equivalence of categories. 

In our context, we need to check that $\Phi_{\mathcal{E}}$ is a morphism of $G$-categories, in the sense of, e.g., \cite[Section 2]{BOver}.  
We do this by showing that the $G$-category structure on $\mathrm{D}^b (C^{\vee})$, given by the geometric action of $G$ on $C^{\vee}$, and the $G$-category structure on $\mathcal{A}_X$ as the left orthogonal to the exceptional sequence of $G$-linearized objects $\mathcal{U}_+, \cO_X$,  coincide. This follows directly from the fact that the Fourier-Mukai kernel bundle $\mathcal{E}$ is a $G$-linearized vector bundle on $X \times C^{\vee}$.  The easiest way to see this is to use the explicit description of $\mathcal{E}$ in \cite[Constructions in \S 2 and Corollary 2.5]{Kuz05}. Indeed, denoting by $\mathcal{U}_+^X$ and $\mathcal{U}_{-}^{C^{\vee}}$ the pullbacks of the tautological subbundles from $\mathrm{LGr}_+(V)$ and $\mathrm{LGr}_{-}(V)$ to 
$$
X\times C^{\vee}\subset \mathrm{LGr}_+(V) \times \mathrm{LGr}_-(V),
$$
the bundle $\mathcal{E}$ is cokernel of the morphism
\[
\xi \colon \mathcal{U}_{-}^{C^{\vee}} \hookrightarrow V\otimes \mathcal{O}_{X\times C^{\vee}} \simeq V^* \otimes \mathcal{O}_{X\times C^{\vee}} \to (\mathcal{U}_+^X)^{\vee},
\]
where the isomorphism in the middle is given by the quadratic form and the other maps are the canonical inclusion and surjection. 

\medskip

In summary, the genus 7 $G$-Fano threefold $X$ furnishes another example with a nonlinearizable action, where all pieces in a semiorthogonal decomposition of the derived category are ``geometric", i.e., equivalent as $G$-categories to derived categories of $G$-varieties (of dimension $\le 1$). Thus the pieces of these decompositions fail to detect the nonlinearizability of $X$.

\bibliographystyle{alpha}
\bibliography{derbirDP}

\end{document}